\documentclass[11pt,reqno]{amsart}
\usepackage[utf8]{inputenc}
\usepackage{amsfonts,amsthm,amsmath,amssymb,thmtools}
\usepackage{booktabs}
\usepackage{boldline,makecell}
\usepackage{array,diagbox}
\usepackage{tikz}
\usetikzlibrary{knots}
\usepackage{graphicx}
\usepackage{enumerate}
\usepackage{hyperref}
\usepackage{color}
\usepackage{tikz-cd}
\usepackage{transparent}
\usepackage[margin=1in]{geometry}
\usepackage[
    maxbibnames=99,
    backend=biber,
    style=alphabetic,
    giveninits=true
]{biblatex}
\DeclareFieldFormat{pages}{#1}
\renewbibmacro{in:}{%
  \ifentrytype{article}
    {}
    {\bibstring{in}%
     \printunit{\intitlepunct}}}
\DeclareFieldFormat
[article,inbook,incollection,inproceedings,patent,thesis,unpublished]
  {title}{\mkbibemph{#1}}
\DeclareFieldFormat{journaltitle}{#1\isdot}
\DeclareFieldFormat[article]{volume}{\mkbibbold{#1}}
\DeclareFieldFormat[article]{number}{\bibstring{number}\addnbspace #1}

\renewbibmacro*{journal+issuetitle}{%
  \usebibmacro{journal}%
  \setunit*{\addspace}%
  \iffieldundef{series}
    {}
    {\newunit
     \printfield{series}%
     \setunit{\addspace}}%
  \printfield{volume}%
  \setunit{\addspace}%
  \usebibmacro{issue+date}%
  \setunit{\addcomma\space}%
  \printfield{number}%
  \setunit{\addcolon\space}%
  \usebibmacro{issue}%
  \setunit{\addcomma\space}%
  \printfield{eid}
  \newunit}

\newtheoremstyle{mystyle}
  {}
  {}
  {\itshape}
  {}
  {\bfseries}
  {.}
  { }
  {\thmname{#1}\thmnumber{ #2}\thmnote{ (#3)}}

\theoremstyle{mystyle}
\newtheorem{Thm}{Theorem}

\newtheorem{Conj}[Thm]{Conjecture}

\theoremstyle{definition}

\theoremstyle{remark}

\declaretheoremstyle[%
  spaceabove=3pt,
  spacebelow=10pt,
  headfont=\normalfont\itshape,%
  postheadspace=.5em,%
  qed=\qedsymbol%
]{mystyle2}

\newcommand{\Z}{\mathbb{Z}}

\title{Cosmetic surgery on satellite knots}

\author{Qiuyu Ren}
\address{Department of Mathematics, University of California, Berkeley, Berkeley, CA 94709, USA}
\email{qiuyu\_ren@berkeley.edu}
\begin{document}

\begin{abstract}
We show that if there exists a knot in $S^3$ that admits purely cosmetic surgeries, then there exists a hyperbolic one with this property.
\end{abstract}

\maketitle

\section{Introduction}
For a knot $K\subset S^3$ and two different slopes $r,r'\in\mathbb{QP}^1$, the surgeries $S^3_r(K)$ and $S^3_{r'}(K)$ are \textit{purely cosmetic} if they are orientation-preservingly homeomorphic. The cosmetic surgery conjecture in $S^3$ states that there are no purely cosmetic surgeries if $K$ is nontrivial.

\begin{Conj}[{\cite[Conjecture~6.1]{gordon1990dehn}\cite[Problem~1.81A]{kirby1997problems}}]\label{conj:cos}
If $K\subset S^3$ is a nontrivial knot, $r\ne r'$, then $S^3_r(K)$ and $S^3_{r'}(K)$ are not purely cosmetic.
\end{Conj}\vspace{-8pt}

Significant progress has been made in the past two decades. Notably, the collective work of Ni-Wu \cite{ni2015cosmetic}, Hanselman \cite{hanselman2022heegaard}, Daemi-Eismeier-Lidman \cite{daemi2024filtered} and others reduces the conjecture to the following special case.

\begin{Thm}[{\cite[Corollary~1.4]{daemi2024filtered}}]\label{thm:strong}
If $K$ is a nontrivial knot, $r\ne r'$, with $S^3_r(K)$ and $S^3_{r'}(K)$ purely cosmetic, then $\{r,r'\}=\{\pm2\}$, $K$ has Seifert genus $2$ and Alexander polynomial $1$.
\end{Thm}\vspace{-8pt}

Other evidence for Conjecture~\ref{conj:cos} abounds. To list a few, the conjecture holds for torus knots, cable knots \cite{tao2019cable}, composite knots \cite{tao2022knots}, knots up to $19$ crossings \cite{futer2024excluding}, and hyperbolic knots with $r$ an exceptional slope \cite{ravelomanana2016exceptional}\cite[Remark~1.7]{daemi2024filtered}. We do not attempt to give a complete survey and refer readers to \cite{daemi2024filtered} for more information.

In this short note, we prove the following theorem. The proof is topological, and exploits existing results mentioned above.

\begin{Thm}\label{thm:main}
Conjecture~\ref{conj:cos} holds if and only if it holds for hyperbolic surgeries on hyperbolic knots.
\end{Thm}\vspace{-8pt}

In fact, we can prove a stronger theorem.
\begin{Thm}\label{thm:satellite}
Suppose $K$ is a satellite knot that admits purely cosmetic surgeries. Let $X$ be the component of the JSJ decomposition of $S^3\backslash K$ adjacent to $K$ and let $V$ be $X$ with $K$ refilled in. Then $V\cong\#^n(S^1\times D^2)$ for some $n$, and $K$ as a knot in $V$ is null-homologous, hyperbolic, and admits two different hyperbolic surgeries that are orientation-preservingly homeomorphic. In particular, there exists infinitely many hyperbolic knots in $S^3$ that admit purely cosmetic hyperbolic surgeries.
\end{Thm}\vspace{-8pt}

\subsection*{Acknowledgement}
I want to thank my advisor Ian Agol for his continued guidance, support, and many stimulating discussions. I also thank Yi Ni for helpful discussions and the referee for their constructive feedback. This work is partially supported by the Simons Investigator Award 376200.

\section{The proof}
\begin{proof}[Proof of Theorem~\ref{thm:main}]
Suppose Conjecture~\ref{conj:cos} does not hold for some nontrivial knot $K$, which means that $S^3_{+2}(K)$ and $S^3_{-2}(K)$ are purely cosmetic by Theorem~\ref{thm:strong}. We take such $K$ so that the number of JSJ tori in the JSJ decomposition of its complement is minimal. It is well-known that torus knots admit no cosmetic surgery (for example, because they have nontrivial Alexander polynomial, cf. Theorem~\ref{thm:strong}). Thus, to prove the theorem, in view of \cite[Remark~1.7]{daemi2024filtered} which says hyperbolic knots admit no exceptional cosmetic surgeries, it suffices to show that $S^3\backslash K$ has trivial JSJ decomposition. 

Assume, to the contrary, that $S^3\backslash K$ has nontrivial JSJ decomposition. We first examine the JSJ decompositions of $S^3_{\pm2}(K)$. Note that by the distance bound on reducible Dehn fillings \cite{gordon1996reducible}, $S^3_{\pm2}(K)$ are both irreducible. 

\textbf{Claim 1}: Every JSJ torus in $S^3\backslash K$ is incompressible in $S^3_{\pm2}(K)$.\vspace{-10pt}
\begin{proof}
Every JSJ torus $T$ exhibits $K$ as a satellite knot $K=P(C)$ for some pattern $P$ and companion $C$. Suppose $K$ has a nontrivial surgery in which $T$ compresses. By \cite[Theorem~2.13]{boyer2025jsj} (which was attributed to the combined work of \cite{berge1991knots,gabai1989surgery,gabai19901,gordon1983dehn,scharlemann1990producing}), $P$ is either a cable pattern or a $1$-bridge braid pattern with winding number $w(P)\ge5$. Cable knots have no purely cosmetic surgeries \cite{tao2019cable}, so $P$ is $1$-bridge with $w(P)\ge5$. The Seifert genus of $K$ is expressed in terms of $P,C$ as\vspace{-2pt}
\begin{equation}\label{eq:satellite_genus}
g(K)=w(P)g(C)+g(P)\vspace{-1pt}
\end{equation}
where $g(P)$ is the relative Seifert genus of $P$, defined as the minimal genus of an embedded orientable surface in $S^1\times D^2$ cobounding $P$ and $w(P)$ canonical longitudes on the boundary torus \cite{schubert1953knoten}. Here we do not orient knots, so for us $w(P)$ is always a nonnegative integer. In particular, $g(K)\ge w(P)\ge5$, contradicting Theorem~\ref{thm:strong}.
\end{proof}

Let $X$ denote the $3$-manifold piece in the JSJ decomposition of $S^3\backslash K$ that is adjacent to $K$.

\textbf{Claim 2}: $X$ is hyperbolic.\vspace{-10pt}
\begin{proof}
If not, $K$ is either a cable knot or a composite knot (see e.g. \cite[Theorem~2(1)]{budney2006jsj}). Such knot has no purely cosmetic surgery by \cite{tao2019cable,tao2022knots}.
\end{proof}

\textbf{Claim 3}: One of the slopes $\pm2$ on the cusp corresponding to $K$ in the hyperbolic manifold $X$ is not exceptional, meaning that the corresponding Dehn filling $X_{+2}$ or $X_{-2}$ is hyperbolic.\vspace{-10pt}
\begin{proof}
Assume, to the contrary, that neither of $X_{\pm2}$ is hyperbolic. Since the slopes $\pm2$ have intersection number $4$ and $X$ has at least two cusps, by work on exceptional Dehn fillings, especially \cite{gordon1999toroidal,gordon2000annular,gordon2008toroidal}, \cite[Table~1]{gordon2000annular2} and references therein, this implies that $X$ is one of three particular hyperbolic manifolds $M_1,M_2,M_{14}$ in the notation of \cite[Theorem~1.1]{gordon2008toroidal}. Here $M_1$ is the Whitehead link complement, $M_2$ is the complement of a certain $2$-component link with linking number $3$ between its two unknotted components, and $M_{14}$ is not a link complement in $S^3$ \cite[Lemma~24.2]{gordon2008toroidal}. Since $X$ is homeomorphic to a link complement by the Fox re-embedding theorem, $X\not\cong M_{14}$. If $X\cong M_1$, \cite[Lemma~24.1(1)]{gordon2008toroidal} further shows that one of $r,r'$ is the $0$ slope (this is well-defined for cusps in $M_1$, characterized by being null-homologous in $M_1$), contradicting $\{r,r'\}=\{\pm2\}$. If $X\cong M_2$, then $K$ is a satellite knot $P(C)$ for some pattern $P$ with $w(P)=3$, which implies $g(K)\ge w(P)=3$ by \eqref{eq:satellite_genus}, contradicting Theorem~\ref{thm:strong}.
\end{proof}

\textbf{Claim 4}: $X_{\pm2}$ are orientation-preservingly homeomorphic (thus both hyperbolic).\vspace{-10pt}
\begin{proof}
One of $X_{\pm2}$, say $X_{+2}$, is hyperbolic by Claim 3. By Claim 1, the JSJ tori of $S^3_{+2}(K)$ are exactly those of $S^3\backslash K$.

The JSJ decomposition of $S^3_{-2}(K)$ is obtained by removing redundant incompressible tori in the set $\{\text{JSJ tori of $S^3\backslash K$}\}\sqcup\{\text{JSJ tori of $X_{-2}$}\}$, where a torus (necessarily a boundary torus of $X$) is redundant if and only if its two adjacent pieces are Seifert fibered, with matching fibrations on the torus. Thus, by comparing the collection of hyperbolic pieces in the JSJ decompositions of $S^3_{\pm2}(K)$, we see that the JSJ decomposition of $X_{-2}$ contains a single hyperbolic piece, which is orientation-preservingly homeomorphic to $X_{+2}$. It remains to show that $X_{-2}$ contains no Seifert fibered piece.

Each Seifert fibered piece in the JSJ decomposition of $S^3\backslash K$ (equivalently, of $S^3_{+2}(K)$) is either the complement of some $(p,q)$-torus knot in $S^3$, the complement of some $(p,q)$-cable pattern in $S^1\times D^2$, or $P_n\times S^1$ where $P_n$ is the planar surface with $n$ boundaries, $n\ge3$ (see e.g. \cite[Theorem~2(1)]{budney2006jsj}). These three types of Seifert fibered manifolds admit unique Seifert fibrations over base orbifolds $(D^2,p,q)$ ($p,q\ge2$ are coprime), $(S^1\times I,p)$ ($p\ge2$), $P_n$ ($n\ge3$), respectively (see e.g. \cite[Proposition~10.4.16]{martelli2016introduction} for the uniqueness).

Assume, to the contrary, that $X_{-2}$ contains Seifert fibered pieces. By comparing the numbers of Seifert fibered pieces in $S^3_{\pm2}(K)$, we know that there exists some Seifert fibered piece $N_0$ in $X_{-2}$ that combines with some adjacent Seifert fibered pieces $N_1,\cdots,N_k$ in $S^3\backslash K$, into a single Seifert fibered piece $N=\cup_{i=0}^kN_i$ in $S_{-2}^3(K)\cong S_{+2}^3(K)$. By assumption, the base orbifolds of $N,N_1,\cdots,N_k$ are each of the form $(D^2,p,q)$, $(S^1\times I,p)$, or $P_n$, as listed above. Since the base orbifold of $N$ is the union of those of the $N_i$'s along common boundaries, the only possibility is that all $N_i$'s and $N$ are of the form $P_n\times S^1$, for various $n\ge3$. Consequently, the collection of $P_n\times S^1$ pieces in $S^3_{-2}(K)$ has a strictly larger complexity than that of $S^3_{+2}(K)$, a contradiction. Here, one can take the complexity as the negative of the sum of euler characteristics of the base surfaces.
\end{proof}

Let $\Gamma$ denote the JSJ graph of $S^3\backslash K$, which is canonically isomorphic to those of $S^3_{\pm2}(K)$ by Claim 4, and is a tree for topological reasons. Let $*$ denote the vertex of $\Gamma$ that corresponds to $X$. Fix an orientation-preserving homeomorphism $\Phi\colon S^3_{+2}(K)\xrightarrow{\cong}S^3_{-2}(K)$, which induces a tree automorphism $\phi$ of $\Gamma$.

\textbf{Claim 5}: $\Gamma$ is the convex hull of $\{\phi^{(i)}(*)\colon 0\le i<N\}$, where $N$ is the minimal positive integer with $\phi^{(N)}(*)=*$.\vspace{-10pt}
\begin{proof}
Let $\Gamma'$ denote the convex hull of $\{\phi^{(i)}(*)\colon0\le i<N\}$. Then, $\phi$ restricts to an automorphism on the subtree $\Gamma'$ and cyclically permutes $*,\phi(*),\cdots,\phi^{(N-1)}(*)$. Thus, the leaves of $\Gamma'$ are exactly $*,\phi(*),\cdots,\phi^{(N-1)}(*)$. Assume, to the contrary, that $\Gamma'$ is a proper subtree of $\Gamma$. Let $(S^3\backslash K)'$ and $S^3_{\pm2}(K)'$ denote the submanifolds of $S^3\backslash K$ and $S^3_{\pm2}(K)$, respectively, that lie above $\Gamma'$.

Let $e_1,\cdots,e_M$ denote edges in $\Gamma\backslash\Gamma'$ adjacent to some $\phi^{(i)}(*)$, where $M=(d-1)N$, $d$ being the degree of $*$ in $\Gamma$. These edges are in canonical one-one correspondence with connected components of $\Gamma\backslash\Gamma'$, each of which lies under a submanifold $M_i\subset S^3\backslash K$, $i=1,\cdots,M$. The homeomorphism $\Phi$ restricts to a permutation of the $M_i$'s. Since the $M_i$'s are nontrivial knot complements by Alexander's theorem, the boundary torus of each $M_i$ has a canonical parametrization $(m_i,\ell_i)$ up to overall sign, where $m_i,\ell_i\subset\partial M_i$ are the meridian, canonical longitude of the corresponding knot, respectively. These parametrizations are respected by the homeomorphisms $\Phi|_{\partial M_i}$ up to signs.

Apply the Fox re-embedding theorem to $(S^3\backslash K)'$ (see e.g. \cite[Proposition~3,Proposition~2]{budney2006jsj}), we obtain an orientation-preserving homeomorphism 
\begin{equation}\label{eq:Fox}
(S^3\backslash K)'\xrightarrow{\cong}S^3\backslash(K_*\cup L),
\end{equation}
where
\begin{enumerate}[(i)]
\item $K_*$ is a knot, $L$ is an unlink with components $U_1,\cdots,U_M$;
\item The cusp corresponding to $K$ is mapped to the cusp corresponding to $K_*$, with parametrization;
\item For each $i$, the cusp $\partial M_i$ is mapped to the cusp corresponding to $U_i$, where $m_i$ is mapped to $\ell_i'$, the canonical longitude of $U_i$, and $\ell_i$ is mapped to $m_i'$, the canonical meridian of $U_i$.
\end{enumerate}

Pick an integer $n>0$ and perform $(1/n)$-Dehn fillings on cusps of $S^3\backslash(K_*\cup L)$ corresponding to components in $L$, the knot $K_*$ becomes a new knot $K'=K'_n\subset S^3$, and the homeomorphism $\Phi|_{(S^3_{+2}(K))'}$, under the identification \eqref{eq:Fox}, extends to an orientation-preserving homeomorphism $S^3_{2-R}(K')\xrightarrow{\cong}S^3_{-2-R}(K')$, where $R=n\sum_i\ell k(K_*,U_i)^2\ge0$. It follows that $K'$ admits purely cosmetic surgeries, and by Theorem~\ref{thm:strong} that $R=0$, hence $\ell k(K_*,U_i)=0$ for all $i$.

The JSJ pieces in $(S^3\backslash K)'\cong S^3\backslash(K_*\cup L)$ adjacent to the cusps that we fill are hyperbolic, since they are homeomorphic to either $X$ or $X_{+2}\cong X_{-2}$. Hence, by Thurston's hyperbolic Dehn filling theorem, for large enough $n$, all these pieces remain hyperbolic after Dehn fillings. It follows that, for large $n$, $K'$ is nontrivial and the JSJ graph of $S^3\backslash K'$ is canonically isomorphic to $\Gamma'$. This yields a contradiction with the minimality of $K$.
\end{proof}

\textbf{Claim 6}: Every JSJ torus $T$ of $S^3\backslash K$ bounds a solid torus $V$ in $S^3$ containing $K$ by Alexander's theorem. If $V$ does not contain the JSJ piece of $S^3\backslash K$ corresponding to the vertex $\phi^{(N-1)}(*)$, then the winding number of $K$ in $V$ is $1$.\vspace{-10pt}
\begin{proof}
Let $V_{+2}(K)$ be the $(+2)$-surgery on $V$ along $K$. By assumption, the homeomorphism $\Phi$ maps $V_{+2}(K)$ into $S^3\backslash K$, hence $V_{+2}(K)$ is homeomorphic to a knot complement by Alexander's theorem. If the winding number of $K$ in $V$ is even, then $H_1(V_{+2}(K);\Z/2)\cong(\Z/2)^2$, a contradiction. Therefore the winding number is odd, which must be $1$ by \eqref{eq:satellite_genus} since the $g(K)=2$ by Theorem~\ref{thm:strong}.
\end{proof}

\textbf{Claim 7}: The JSJ graph $\Gamma$ contains exactly $2$ vertices $*,\phi(*)$.\vspace{-10pt}
\begin{proof}
Otherwise, by Claim 5, the distance between $*$ and $\phi^{(N-1)}(*)$ is at least $2$, meaning that there are at least two JSJ tori in $S^3\backslash K$ separating $K$ and the JSJ piece corresponding to $\phi^{(N-1)}(*)$. Therefore we may write $K=P_1(P_2(C))$ for some nontrivial satellite patterns $P_1,P_2$ and companion $C$, where $w(P_1)=w(P_2)=1$ by Claim 6. A nontrivial winding number $1$ satellite pattern has nonzero relative Seifert genus, hence $g(K)\ge g(P_2(C))+1\ge g(C)+2>2$ by \eqref{eq:satellite_genus}, contradicting Theorem~\ref{thm:strong}.
\end{proof}

\textbf{Claim 8}: We have a contradiction.\vspace{-10pt}
\begin{proof}
By Claim 7, $S^3\backslash K$ (resp. $S^3_{\pm2}(K)$) contains a unique JSJ torus $T$, which exhibits $K$ as a satellite knot $P(J)$. Let $V$ denote the solid torus in $S^3$ bounded by $T$. In $S^3\backslash K$ (resp. $S^3_{\pm2}(K)$), the JSJ piece corresponding to $*$ is $X=V\backslash K$ (resp. $X_{\pm2}=V_{\pm2}(K)$) and the piece corresponding to $\phi(*)$ is $S^3\backslash J$. Let $(m,\ell)$ be a parametrization of $T$ by the meridian and the canonical longitude of $J$ (canonical up to overall sign) as before, where $m,\ell$ have oriented intersection $+1$ when $T$ is oriented as the boundary of $V$.

\begin{figure}
	\centering
	\begin{tikzpicture}[thick,scale=0.8, every node/.style={scale=0.8}]
		\begin{scope}[xshift=-100pt]
			
			\begin{knot}[
				clip width=8,
				clip radius=8pt,
				flip crossing/.list={2,4,6},
				]
				\node (A) at (0,0) [draw,minimum width=60pt,minimum height=25pt,thick] {$P$};
				\strand [thick] (-4,0)
				to[out=270, in=180] (-3,-0.5)
				to[out=0, in=270] (-2,0)
				to[out=90,in=0] (-3,0.5)
				to[out=180,in=90] (-4,0);
				\strand [thick] (A.90)
				to[out=90, in=0] (-1.5,2)
				to[out=180, in=90] (-3,0)
				to[out=270, in=180] (-1.5,-2)
				to[out=0,in=270] (A.270);
				\strand [thick] (A.50)
				to[out=90, in=0] (-1.6,2.4)
				to[out=180, in=90] (-3.4,0)
				to[out=270, in=180] (-1.6,-2.4)
				to[out=0,in=270] (A.310);
				\strand [thick] (A.130)
				to[out=90, in=0] (-1.4,1.6)
				to[out=180, in=90] (-2.6,0)
				to[out=270, in=180] (-1.4,-1.6)
				to[out=0,in=270] (A.230);
				\node at (0.2,2) {$-2$};
				\node at (-4.3,0.4) {$-1$};
			\end{knot}
		\end{scope}
		\node at (-1.35,0) {\Huge $\sim$};
		\begin{scope}[xshift=100pt]
			\begin{knot}[
				clip width=8,
				clip radius=8pt,
				flip crossing/.list={2,4,6},
				]
				\node (A) at (0,0) [draw,minimum width=60pt,minimum height=25pt,thick] {$P$};
				\node (B) at (-3,0) [draw,minimum width=50pt,minimum height=25pt,thick] {$+1$};
				\strand [thick] (A.90)
				to[out=90, in=0] (-1.5,2)
				to[out=180, in=90] (B.90);
				\strand [thick] (B.270)
				to[out=270, in=180] (-1.5,-2)
				to[out=0,in=270] (A.270);
				\strand [thick] (A.50)
				to[out=90, in=0] (-1.6,2.4)
				to[out=180, in=90] (B.130);
				\strand [thick] (B.230)
				to[out=270, in=180] (-1.6,-2.4)
				to[out=0,in=270] (A.310);
				\strand [thick] (A.130)
				to[out=90, in=0] (-1.4,1.6)
				to[out=180, in=90] (B.50);
				\strand [thick] (B.310)
				to[out=270, in=180] (-1.4,-1.6)
				to[out=0,in=270] (A.230);
				\node at (0.2,2) {$-1$};
			\end{knot}
		\end{scope}
	\end{tikzpicture}
	\caption{A Kirby move showing that the Dehn filling of $X_{-2}$ by $\ell-m$ is homeomorphic to the $(-1)$-surgery on $P_{+1}(U)$.}
	\label{fig:Kirby}
\end{figure}

Since the winding number of the satellite pattern $P$ is $1$ by Claim 6, $X=V\backslash K$ is a homology $T^2\times I$ with first homology generated by $m$ and $\ell$, and $m,\ell$ are homologous to the meridian, the canonical longitude of $K$ on the boundary cusp of $S^3\backslash K$, respectively. Consequently, $\ell\pm2m$ gives the only slope on $T$ that is null-homologous in $X_{\pm2}$, for each sign $\pm$ respectively. Since the homeomorphism $\Phi$ restricts to orientation-preserving homeomorphisms $X_{+2}\xrightarrow{\cong}S^3\backslash J$, $S^3\backslash J\xrightarrow{\cong}X_{-2}$, we conclude that $\Phi|_T(\ell+2m)=\epsilon\ell$ and $\Phi|_{T}(\ell)=\epsilon'(\ell-2m)$ for some signs $\epsilon,\epsilon'$. Since $\Phi|_T\colon T\xrightarrow{\cong}T$ is an orientation-reversing homeomorphism, we have $\epsilon'=-\epsilon$. It follows that $\Phi|_T(m)=\pm(\ell-m)$, which implies that the Dehn filling of $X_{-2}$ by the slope $\ell-m$ is $S^3$. But this Dehn filling is also the $(-1)$-surgery on $P_{+1}(U)$ ($P_n$ denotes the $n$-twisted pattern $P$), as can be seen from Figure~\ref{fig:Kirby}, so $P_{+1}(U)$ is the unknot by property P \cite[Theorem~2]{gordon1989knots}. Similarly, from $\Phi|_T^{-1}(m)=\pm(\ell+m)$, we conclude that $P_{-1}(U)$ is the unknot. We have thus shown that the Dehn fillings of $X$ on the cusp $T$ by slopes $\ell\pm m$ are both $S^1\times D^2$. However, \cite[Theorem~1]{wu1992incompressibility} implies that any two slopes on the same cusp of $X$ with boundary-reducible fillings have intersection number at most $1$, a contradiction.
\end{proof}
The proof of Theorem~\ref{thm:main} is complete.
\end{proof}

\begin{proof}[Proof of Theorem~\ref{thm:satellite}]
The proof of Theorem~\ref{thm:main} mostly applies, except that we do not get a contradiction if Claim 5 fails.

Use the notations as before. If $\phi(*)\ne*$, the argument in Claim 5 shows that by replacing $K$ by another satellite knot $K'$ if necessary, we may assume $\Gamma$ is the convex hull of $\{\phi^{(i)}(*)\colon0\le i<N\}$. Then the rest of the argument carries through and still yields a contradiction. Therefore we must have $\phi(*)=*$, meaning that $\Phi$ restricts to a homeomorphism $X_{+2}\xrightarrow{\cong}X_{-2}$. Let $V$ denote $X$ with $K$ refilled in; thus, $K$ admits purely cosmetic hyperbolic surgeries as a knot in $V$. Under the Fox re-embedding identification \eqref{eq:Fox}, $V\cong S^3\backslash L\cong\#^M(S^1\times D^2)$, and the fact that $K$ is null-homologous in $V$ follows from the conclusion $\ell k(K_*,U_i)=0$ in the proof of Claim 5. We obtain infinitely many hyperbolic knots in $K'_n\subset S^3$ with purely cosmetic surgeries by performing $(1/n)$-surgeries on the components of $L$ as in the proof of Claim 5 (the infinitude follows from the fact that the hyperbolic volumes of $K_n'$ converge to that of $V\backslash K$ strictly from below as $n\to\infty$).
\end{proof}

\printbibliography

\end{document}